\documentclass[12pt, a4paper, leqno]{amsart}
\usepackage[utf8]{inputenc}
\usepackage{amsmath}
\usepackage{amsfonts}
\usepackage{amssymb}
\usepackage{times}
\usepackage[T1]{fontenc} 
\usepackage{url}
\usepackage[dvipsnames]{xcolor}
\usepackage[colorlinks, allcolors=RedViolet,pdfstartview=,pdfpagemode=UseNone]{hyperref} 
\usepackage{graphicx}
\usepackage{esint}
\usepackage{enumitem}

\usepackage{pgf,tikz}
\usepackage{mathrsfs}
\usetikzlibrary{arrows}

\let\oldmarginpar\marginpar
\renewcommand\marginpar[1]{\-\oldmarginpar[\raggedleft\footnotesize #1]%
{\raggedright\footnotesize #1}}

\usepackage{amsthm}
\theoremstyle{plain}
\newtheorem{thm}[equation]{Theorem}
\newtheorem{lem}[equation]{Lemma}

\theoremstyle{definition}
\newtheorem{defn}[equation]{Definiton}

\theoremstyle{remark}

\newtheorem{open}[equation]{Question}

\numberwithin{equation}{section}

\newcommand{\R}{\mathbb{R}}

\newcommand{\Q}{\mathbb{Q}}
\newcommand{\Rn}{\mathbb{R}^n}

\renewcommand{\phi}{\varphi}
\renewcommand{\rho}{\varrho}
\renewcommand{\epsilon}{\varepsilon}
\renewcommand{\vartheta}{\theta}

\def\le{\leqslant}

\def\ge{\geqslant}

\def\diam{\qopname\relax o{diam}}

\def\px{{p(\cdot)}}

\def\phix{{\phi}}

\def\psix{{\psi}}

\def\Mf{{M\!f}}

\newcommand{\Phiw}{\Phi_{\rm{w}}}

\newcommand{\ainc}[1]{\hyperref[defn:aIncp]{{\normalfont(aInc){\ensuremath{_{#1}}}}}}
\newcommand{\adec}[1]{\hyperref[defn:aDecq]{{\normalfont(aDec){\ensuremath{_{#1}}}}}}
\newcommand{\inc}[1]{\hyperref[defn:Inc]{{\normalfont(Inc){\ensuremath{_{#1}}}}}}
\newcommand{\dec}[1]{\hyperref[defn:Dec]{{\normalfont(Dec){\ensuremath{_{#1}}}}}}

\newcommand{\azero}{\hyperref[defn:A0]{{\normalfont(A0)}}}
\newcommand{\aone}{\hyperref[defn:A1]{{\normalfont(A1)}}}
\newcommand{\atwo}{\hyperref[defn:A2N]{{\normalfont(A2)}}}

\newcommand{\aoneo}[1]{\hyperref[defn:A1Omega]{{\normalfont(A1){\ensuremath{_{#1}}}}}}
\date{\today}

\usepackage{color}
\definecolor{blau}{rgb}{0.1,0.0,0.9}

\newcounter{komcounter}
\numberwithin{komcounter}{section}

\begin{document}

\title{Extension in generalized Orlicz spaces}

\author{Petteri Harjulehto}
 \address{Petteri Harjulehto,
 Department of Mathematics and Statistics,
FI-20014 University of Turku, Finland}
\email{\texttt{petteri.harjulehto@utu.fi}}

\author{Peter Hästö}
 \address{Peter Hästö, Department of Mathematics and Statistics,
FI-20014 University of Turku, Finland, and Department of Mathematics,
FI-90014 University of Oulu, Finland}
\email{\texttt{peter.hasto@oulu.fi}}





\begin{abstract}
We prove that a $\Phi$-function can be extended from a domain $\Omega$ 
to all of $\Rn$ while preserving crucial properties for harmonic 
analysis on the generalized Orlicz space $L^\Phi$. 
\end{abstract}

\maketitle

\section{Introduction}

Generalized Orlicz spaces, also known as Musielak--Orlicz spaces, 
and related differential equations have been studied with increasing 
intensity recently, see, e.g., the references 
\cite{AhmCGY18, BarCM18, ByuO17, Chl18, CruH18, DefM_pp, EleMM18, Kar18, MizS_pp, OhnS_pp}
published since 2018. This year, we published the monograph 
\cite{HarH_book} in which we present a new framework for the basics 
of these spaces. In contrast to earlier studies, we emphasize properties 
which are invariant under equivalence of $\Phi$-functions. This 
means, in particular, that we replace convexity by the assumption 
that $\frac{\phi(t)}t$ be almost increasing. Within this framework 
we can more easily use techniques familiar from the $L^p$-context, 
see for instance the papers \cite{HarHL_pp, HasO_pp} with applications 
to PDE. 

In the pre-release version of the book \cite{HarH_book}, we had included 
a section on extension of the $\Phi$-function. However, this 
was removed from the final, published version, as we could not at that 
point prove a result with which we were satisfied. In this article we remedy 
this short-coming. 

In many places, extension offers an easy way to 
prove result in $\Omega\subset \Rn$ from results in $\Rn$. For example, 
in variable exponent spaces we may argue 
\[
\| \Mf \|_{L^\px(\Omega)} 
\le 
\| \Mf \|_{L^\px(\Rn)} 
\le 
\| f \|_{L^\px(\Rn)} 
=
\| f \|_{L^\px(\Omega)} 
\]
where $f$ is a zero extension of $f:\Omega\to \R$ to $\Rn$, provided that 
$p:\Omega\to [1,\infty]$ can be extended as well (cf.\ \cite[Proposition~4.1.17]{DieHHR11}) and the maximal operator is bounded in $\Rn$. 

In generalized Orlicz spaces this is not as simple. 
In \cite{HarH_book} we showed that many results of harmonic analysis 
hold in $L^\phi$ provided $\phi$ satisfies conditions \azero{}, \aone{} and \atwo{} (see the 
next section for definitions).
However, the extension requires a stronger condition than \aone{}, namely \aoneo{\Omega}, 
which is known to be equivalent with \aone{} only in quasiconvex domains (Lemma~\ref{lem:quasiconvex}). 
Furthermore, we show in the main result of this paper, Theorem~\ref{thm:extension}, that \aoneo{\Omega} 
also is necessary. Thus we solve the extension problem for the assumptions of \cite{HarH_book}. 
In some recent papers, e.g. \cite{HasO_pp}, also stronger assumptions have been used to 
obtain higher regularity. Whether the extension can be chosen to preserve 
those conditions remains an open problem. 

\section{Preliminaries}

In this section we introduce the terminology and auxiliary results needed in this paper. 
The notation $f\lesssim g$ means that there exists a constant
$C>0$ such that $f\le C g$. 
The notation $f\approx g$ means that
$f\lesssim g\lesssim f$.

\begin{defn}
A function $g: (0, \infty) \to \R$ is \emph{almost increasing} if there exists a constant $a\ge 1$ such that
$g(s) \le a g(t)$ for all $0<s<t$. \emph{Almost decreasing} is defined analogously.
\end{defn}

Increasing and decreasing functions are included in the 
previous definition as the special case $a=1$.

\begin{defn}
\label{defn:aIncp}\label{defn:aDecq}\label{defn:aInc}\label{defn:aDec}
\label{defn:Dec}\label{defn:Inc}
Let $f: \Omega \times [0, \infty)\to \R$ and $p,q>0$.  We say that $f$ satisfies
\begin{itemize}
\item[(Inc)$_p$] if $\frac{f(x,t)}{t^p}$ is increasing; 
\item[(aInc)$_p$] if $\frac{f(x,t)}{t^p}$ is almost increasing; 
\item[(Dec)$_q$] if $\frac{f(x,t)}{t^q}$ is decreasing; 
\item[(aDec)$_q$] if $\frac{f(x,t)}{t^q}$ is almost decreasing;
\end{itemize}
all conditions should hold for almost every $x\in\Omega$ and the 
almost increasing/decreasing constant should be independent of $x$.
\end{defn}


Suppose that $\phi$ satisfies \ainc{p_1}.
Then it satisfies \ainc{p_2} for $p_2<p_1$ and it does not satisfy \adec{q} for $q<p_1$.  
Likewise, if $\phi$ satisfies \adec{q_1},
then it satisfies \adec{q_2} for $q_2>q_1$ and it does not satisfy \ainc{p} for $p>q_1$.
 
\begin{defn}
\label{def:phi_func}
Let $\phi\colon [0,\infty) \to [0,\infty]$ be increasing 
with $\phi(0)=0$, $\lim_{t \to 0^+} \phi(t)= 0$ and $\lim_{t \to \infty} \phi(t) =
\infty$. We say that such $\phi$ is a
({\em weak}) {\em $\Phi$-function} if it satisfies \ainc{1} on $(0, \infty)$.
The set of weak $\Phi$-functions is denoted by $\Phiw$.
\end{defn}

We mention the epithet ``weak'' only when special emphasis is needed. 
Note that when we speak about $\Phi$-functions, we mean the weak $\Phi$-functions 
of the previous definition, whereas many other authors use this term for 
convex $\Phi$-functions, possibly with additional assumptions as well. 
If $\phi$ is convex and $\phi(0)=0$, then we obtain for $0<s<t$ that
\begin{equation}\label{eq:convex=>inc}
\phi(s) = \phi\Big(\frac st t + 0 \Big)\le 
\frac st \phi(t) + \Big(1-\frac st \Big)\phi(0)=\frac st \phi(t),
\end{equation}
i.e.\ \inc{1} holds. 

\begin{defn} \label{def:class_Phi}
A function $\phi\colon \Omega \times [0,\infty) \to [0,\infty]$ is
said to be a ({\em generalized weak}) {\em $\Phi$-function}, denoted $\phi\in\Phiw (\Omega)$, 
if $x \mapsto \phi(y,|f(x)|)$ is measurable for every measurable $f$,
$\phi(y,\cdot)$ is a weak $\Phi$-function for almost every $y\in \Omega$
and $\phi$ satisfies \ainc{1}.
\end{defn}

Unless there is danger of confusion, we will drop the word ``generalized''.
Note that if  $x \mapsto \phi(y,t)$ is measurable for every $t \ge 0$ and $t \mapsto \phi(x,t)$ is 
left-continuous for almost every $x\in \Omega$, then  $x \mapsto \phi(y,|f(x)|)$ is measurable for every measurable $f$, \cite[Theorem~2.5.4]{HarH_book}.

Two functions $\phi$ and $\psi$ are \textit{equivalent}, 
$\phi\simeq\psi$, if there exists $L\ge 1$ such that 
$\phi(x, \frac tL)\le \psi(x, t)\le \phi(x, Lt)$ for all $x$ and all $t \ge 0$. 
Short calculations show that $\simeq$ is an equivalence relation.
Note that if $\phi\simeq\psi$, then $L^\phix (\Omega)=L^\psix(\Omega)$,
see Theorem 3.2.6 in \cite{HarH_book}. 

Let us define a left-inverse of $\phi$ by
\[
\phi^{-1}(x, \tau) := \inf\{t \ge 0: \phi(x, t)\ge \tau\}. 
\]
Note that $\phi^{-1}$ is left-continuous when $\phi$ is increasing and $\phi(0) =0$. Moreover if $\phi \in \Phiw$, then 
$\phi$ satisfies \ainc{p}
if and only if 
$\phi^{-1}$ satisfies \adec{1/p}; and
$\phi$ satisfies \adec{q}
if and only if 
$\phi^{-1}$ satisfies \ainc{1/q}.
These and other properties can be found in \cite[Chapter~2.3]{HarH_book}.

\begin{defn}\label{defn:A0}  \label{defn:A1}  \label{defn:A2N}
Let $\phi \in \Phiw(\Omega)$. We define three conditions:
\begin{itemize}
\item[(A0)] 
There exists $\beta \in(0, 1]$ such that $ \beta \le \phi^{-1}(x,1) \le \frac1{\beta}$ for almost every $x \in \Omega$.
\item[(A1)] 
There exists $\beta\in (0,1)$ such that
\[
\beta \phi^{-1}(x, t) \le \phi^{-1} (y, t)
\]
for every $t\in [1,\frac 1{|B|}]$, almost every $x,y\in B \cap \Omega$ and 
every ball $B$ with $|B|\le 1$.
\item[(A2)] 
For every $s>0$ there exist $\beta\in(0,1]$ and $h\in L^1(\Omega) \cap L^\infty (\Omega)$ 
such that 
\[
\beta \phi^{-1}(x,t) 
\le 
\phi^{-1}(y,t)
\]
for almost every $x,y\in \Omega$ and every $t\in [h(x)+h(y), s]$.
\end{itemize}
\end{defn}

By \cite[Lemma~4.2.7]{HarH_book} \atwo{} is equivalent with the following condition:
there exist $\phi_\infty\in \Phiw$, 
$h\in L^1(\Omega) \cap L^\infty (\Omega)$, $s>0$ and $\beta\in (0,1]$ such that 
\begin{equation}\label{equ:A2'}
\phi(x,\beta t) \le \phi_\infty(t) + h(x)
\quad\text{and}\quad
\phi_\infty(\beta t) \le \phi(x,t) + h(x)
\end{equation}
for almost every $x\in\Omega$.

Note how the conditions \azero{}--\atwo{} are formulated in terms of the inverse 
function $\phi$. This turns out to be very convenient in many cases, since 
the appropriate range of $t$ for which the comparison can be done is easily 
expressed for the inverse function. We also use the inverse function for the extension. 
In \cite[Proposition~2.5.14]{HarH_book}, we showed that $f$ is the inverse of 
some $\Phi$-function if and only if it satisfies the following conditions:
\begin{enumerate}
\item
$f$ is increasing;
\item
$f$ is left-continuous;
\item
$f$ satisfies \adec{1};
\item
$f(t) = 0$ if and only if $t = 0$,
and, $f(t) = \infty$ if and only if $t = \infty$;
\item
$x\mapsto f(x,t)$ is measurable for all $t\ge 0$
\end{enumerate}


\section{Extension}
\label{sect:extension}

In the following version of \aone{} we can use any size of $t$, but have to pay 
in terms of a smaller constant for large $t$. 

\begin{defn}\label{defn:A1Omega}
Let $\Omega \subset \Rn$. We say that $\phi \in \Phiw(\Omega)$ satisfies
\aoneo{\Omega}, if  there exist a constant $\beta \in(0,1]$  such that
$\beta^{|x-y| t^{1/n} +1} \phi^{-1}(y, t) \le  \phi^{-1}(x, t)$
for all $x, y \in \Omega$ and $t\ge 1$.
\index{generalized Phi-function@generalized $\Phi$-function!(A1z)@\aoneo{\Omega}}%
\index{(A1z)@\aoneo{\Omega}}
\end{defn}

By Theorem~2.3.6 of \cite{HarH_book} we have $\phi \simeq \psi$ if and only if $\phi^{-1} \approx \psi^{-1}$. Hence we obtain the following lemma.

\begin{lem}\label{lem:A1-Omega-invariance}
The condition \aoneo{\Omega} is invariant under equivalence of weak $\Phi$-functions.
\end{lem}

A domain $\Omega \subset \Rn$ is \emph{quasi-convex}, if there exists a constant $K \ge 1$ such that every pair $x, y \in \Omega$ can be connect by  a rectifiable path $\gamma \subset \Omega$ with the
length $\ell(\gamma) \le K|x-y|$. 
\index{quasi-convex!domain}

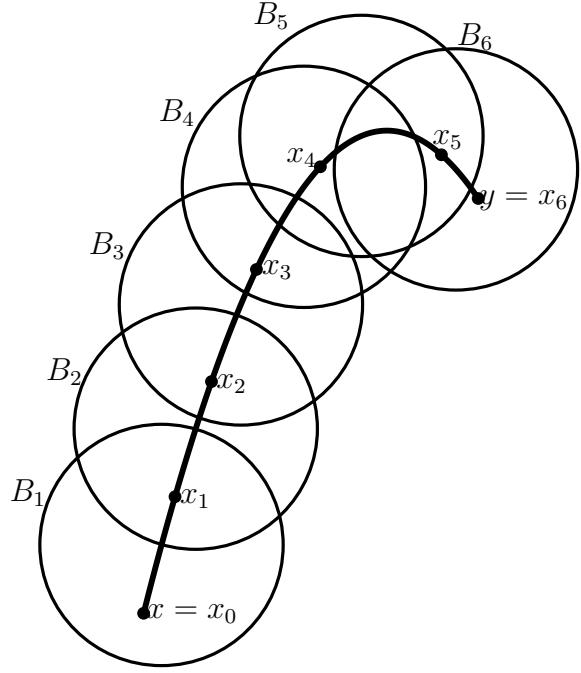
\begin{figure}[ht!]
\begin{center}
\begin{tikzpicture}[line cap=round,line join=round,>=triangle 45,x=1.0cm,y=1.0cm,thick,scale=0.8, every node/.style={scale=1}]
\clip(-5.302968243890198,-5.611321526129977) rectangle (4.357169835778602,6.067634472532111);
\draw [line width=1.2pt] (-2.7061395348837207,-3.367735126014061) circle (2.cm);
\draw [line width=1.2pt] (-2.143348837209302,-1.4403209561925348) circle (2.cm);
\draw [line width=1.2pt] (-1.4004651162790693,0.618883612763657) circle (2.cm);
\draw [line width=1.2pt] (-0.3649302325581387,2.5684827301243924) circle (2.cm);
\draw [line width=1.2pt] (0.5805581395348848,3.4120342628447813) circle (2.cm);
\draw [line width=1.2pt] (2.1338604651162805,2.857180222823146) circle (2.cm);
\draw[line width=2.2pt,color=black,smooth,samples=100,domain=-3.0:2.5] plot(\x,-0.5*\x*\x+\x+3);
\draw [fill=black] (-3.,-4.5) circle (2.5pt);
\draw[color=black] (-2.2,-4.5) node {$x=x_0$};
\draw [fill=black] (-1.8843392856996477,-0.659706557515177) circle (2.5pt);
\draw[color=black] (-1.55,-0.7) node {$x_2$};
\draw [fill=black] (-2.4836220370137783,-2.5678112483840128) circle (2.5pt);
\draw[color=black] (-2.15,-2.6) node {$x_1$};
\draw [fill=black] (-0.09484690811690388,2.900655123893428) circle (2.5pt);
\draw[color=black] (-0.4,3.0582750512290366) node {$x_4$};
\draw [fill=black] (-1.145341162409211,1.198755648436348) circle (2.5pt);
\draw[color=black] (-0.8,1.2) node {$x_3$};
\draw [fill=black] (1.8952690262614031,3.0992466853084792) circle (2.5pt);
\draw[color=black] (2.0,3.35) node {$x_5$};
\draw [fill=black] (2.5,2.375) circle (2.5pt);
\draw[color=black] (3.25,2.35) node {$y=x_6$};
\draw[color=black] (-4.9,-2.5) node {$B_1$};
\draw[color=black] (-4.3,-0.5) node {$B_2$};
\draw[color=black] (-3.6,1.6) node {$B_3$};
\draw[color=black] (-2.45,3.8) node {$B_4$};
\draw[color=black] (-0.9,5.4) node {$B_5$};
\draw[color=black] (2.4515566599064362,5.0582255130354685) node {$B_6$};
\end{tikzpicture}
\caption{Points $x_j$ and balls $B_j$ in the proof of Lemma~\ref{lem:quasiconvex}.}\label{fig:quasiconvex}
\end{center}
\end{figure}

\begin{lem}\label{lem:quasiconvex}
If $\Omega\subset \Rn$ is quasi-convex, 
then $\phi\in \Phiw(\Omega)$ satisfies \aone{} if and only if it satisfies \aoneo{\Omega}.
\end{lem}

\begin{proof}
Assume first that \aoneo{\Omega} holds. Let $B$ be a ball with $|B| \le 1$ and 
$x, y \in \Omega \cap B$. Since $|x-y|\le \diam(B)$, $|x-y|t^{\frac1n} \le c(n)$
for $t \in[1, \tfrac{1}{|B|}]$. Hence \aone{} holds with constant $\beta^{c(n)+1}$.

Assume then that \aone{} holds.
Let $x, y\in\Omega$, $t\ge 1$
and $\gamma\subset\Omega$ be a path connecting $x$ and $y$ 
of length at most $K\,|x-y|$. Let $x_0:=x$ and $\omega_n$ be the measure of the unit ball. Choose points $x_j \in \gamma$ such that 
$\ell(\gamma(x,x_j)) = \frac{j}{( \omega_n t)^{1/n}}$ for $j=1,\ldots, k-1$ 
when possible and finally set $x_{k}=y$. 
Then $|x_{j-1}-x_{j}| \le\frac{1}{(\omega_n t)^{1/n}}$ for all $j$. Let $B_{j+1}$ be an open ball 
such that $x_j, x_{j+1} \in B_{j+1}$ and $\diam(B_{j+1}) = 2 |x_j-x_{j+1}|$, see Figure~\ref{fig:quasiconvex}. 
Then $\frac{1}{|B|} = \frac{1}{\omega_n |x_j-x_{j+1}|^n} \ge t$.
Thus $t$ is in the allowed range for \aone{} and so 
$\beta\phi^{-1}(x_{j+1}, t) \le  \phi^{-1}(x_j, t)$. With this chain of inequalities, we obtain 
that $\beta^k \phi^{-1}(y, t) \le  \phi^{-1}(x, t)$. 
On the other hand, at most
\[
k = \frac{\ell (\gamma(x,x_{k-1}))}{\diam(B)}+1 \le \frac{K|x-y|}{\diam(B)} +1 \le \frac{K|x-y|}{2/ ( \omega_n t)^{1/n}} +1 = c' K t^{1/n}\, |x-y| +1
\] 
points $x_j$ are needed, 
so that $\beta^{c' K t^{1/n} |x-y|+1} \phi^{-1}(y, t) \le  \phi^{-1}(x, t)$ 
for all $x, y \in \Omega$ and $t\ge1$.
\end{proof}

\begin{open}
Does there exist $\Omega\subset\Rn$ and $\phi\in\Phiw(\Omega)$ such 
that \aone{} holds but \aoneo{\Omega} does not?
\end{open}

We say that $\psi\in \Phiw(\Rn)$ is an \textit{extension} of 
$\phi\in \Phiw(\Omega)$ if $\psi|_\Omega \simeq \phi$. 
Since we consider properties which hold up to equivalence of $\Phi$-functions 
this is a natural definition. However, if one wants identity in $\Omega$ 
this is easily achieved by choosing 
$\psi_2 := \phi \chi_\Omega + \psi \chi_{\Rn\setminus\Omega}$, which is 
equivalent to $\psi$ and hence has the same properties.

The next theorem was proved in \cite[Proposition~5.2]{HarHK16} but the proof was incorrect: 
the function $f$ constructed was not increasing, its measurability was unclear, and 
$\psi|_\Omega \simeq \phi$ was shown only for $t \ge 1$.

\begin{thm}\label{thm:extension}
Suppose that $\Omega\subset \Rn$ and $\phi\in \Phiw(\Omega)$.
Then there exists an extension $\psi\in \Phiw(\Rn)$ of $\phi$ which
satisfies \azero, \aone{} and \atwo{}, if and only if $\phi$ satisfies \azero, \aoneo{\Omega}
and \atwo{}.

If $\phi$ satisfies \ainc{p} and/or \adec{q}, then the extension can be 
taken to satisfy it/them, as well.
\end{thm}

\begin{proof}
Suppose first that there exists an extension $\psi\in \Phiw(\Rn)$ which
satisfies \azero, \aone{} and \atwo. Since $\phi\simeq \psi|_\Omega$ we find that $\phi$ 
satisfies \azero{} and \atwo{}.
Let $x,y\in \Omega$. Since $\psi$ satisfies \aone{} and $\Rn$ is quasi-convex, Lemma~\ref{lem:quasiconvex}
implies that $\psi$ satisfies \aoneo{\Rn}, so that 
$\beta^{t^{1/n} |x-y|+1} \psi^{-1}(y, t) \le  \psi^{-1}(x, t)$, 
Since $\psi|_\Omega \simeq \phi$, this implies \aoneo{\Omega} of $\phi$ by Lemma~\ref{lem:A1-Omega-invariance}. 
So we can move on to the converse implication. 

Let $\phi \in\Phiw(\Omega)$ satisfy \azero{}, \aoneo{\Omega}, \atwo{} and \ainc{p} 
for some $p \ge 1$. Then $\phi^{-1}$ satisfies \adec{1/p}. 
Let $\beta_0$ be a constant form \azero{} of $\phi$ 
and $\beta$ be from \aoneo{\Omega} of $\phi$. Next we use \atwo{}.
By \cite[Lemma~4.2.7]{HarH_book}, there exists $\phi_\infty\in \Phiw$ 
and $h\in L^1(\Omega)\cap L^\infty(\Omega)$ such that 
\[
\phi(x,\beta_2 t)\le \phi_\infty(t)+h(x)
\quad\text{and}\quad
\phi_\infty(t) \le \phi(x,\beta_2 t)+h(x)
\]
for almost every $x\in\Omega$ when $\phi_\infty(t)\le \beta_0$ and 
$\phi(x,t)\le \beta_0$, respectively. Note that $\phi_\infty$ is equivalent to 
$\liminf_{x\to\infty}\phi(x,\cdot)$; hence satisfies \azero{} and \adec{1/p} if $\phi$ does. 

Denote $\hat \Omega := \Omega\cap\Q^n$ and 
define $f: \Rn \times [0, \infty] \to [0, \infty]$ by 
\[
f(x,t) := 
\begin{cases}
\beta_0^2 \phi^{-1}(x, t)\chi_\Omega(x) + \beta_0^2 \phi_\infty^{-1}(t) \chi_{\Rn\setminus\Omega}(x)
 \quad &\text{if} \quad  t \in [0,1]; \\
\min \Big\{ (\phi^-_\Omega)^{-1}(t), \inf_{y\in \hat\Omega} \beta^{-|x-y| t^{1/n} } \phi^{-1}(y, t) \Big\}
\quad &\text{if} \quad  t \in (1, \infty);\\
\infty &\text{if} \quad t= \infty.
\end{cases}
\] 

The following properties follow directly from the corresponding properties of 
$\phi^{-1}$ and the definition of $f$:
\begin{enumerate}[label=(\arabic*)]
\item\label{f1} $f(x, t)=0$ iff $t=0$, and $f(x, t)=\infty$ iff $t=\infty$.
\item\label{f5} $f$ satisfies \adec{1/p} in $[0,1]$.
\item\label{f7} $f$ satisfies \azero{}.
\end{enumerate}
For properties \azero{}--\atwo{} we think of $f$ as the inverse in the conditions, 
i.e.\ $\phi^{-1}$ is replaced by $f$, not $f^{-1}$, in the inequalities. 
We next prove the following properties of $f$:
\begin{enumerate}[label=(\arabic*)]
\addtocounter{enumi}{3}
\item\label{f2} $f$ is increasing. 
\item\label{f3} $f$ is left-continuous and measurable. 
\item\label{f6} $f|_\Omega \approx \phi^{-1}$. 
\item\label{f8} $f$ satisfies \aone{}. 
\item\label{f10} $f$ satisfies \ainc{1/q} provided that $\phi$ satisfies \adec{q}.
\end{enumerate}



\underline{Claim \ref{f2}:} 
Since $\phi^{-1}$, $\phi_\infty$ and $(\phi^-_\Omega)^{-1}$ are increasing, 
the claim follows in each subinterval. 
To show that $f(x, 1) \le f(x, t)$ for $1 < t$ we note that
$(\phi^-_\Omega)^{-1}(t) \ge (\phi^-_\Omega)^{-1}(1) \ge \beta_0$ and
$\beta^{-|x-y| t^{1/n} } \phi^{-1}(y, t) \ge \phi^{-1}(y, 1) \ge \beta_0$. 
Thus $f(x, t) \ge \beta_0$. On the other hand, 
$f(x, 1) \le \beta_0^2 (\phi^-_\Omega)^{-1}(1) \le \frac{\beta_0^2}{\beta_0} = \beta_0$.

\underline{Claim \ref{f3}:} 
Since $\phi^{-1}$, $\phi_\infty^{-1}$ and $(\phi^-_\Omega)^{-1}$ are left-con\-tin\-uous and the minimum of left-continuous functions is left-continuous, the claim follows by the definition
of $f$. 
By Lemma~2.5.12 of \cite{HarH_book}, $y \mapsto \phi^{-1}(y,t)$ is measurable. 
The infimum of measurable functions over countable sets is measurable. Thus $x \mapsto f(x, t)$ is measurable for every $t$.


\underline{Claim \ref{f6}:} 
We show that $f \approx \phi^{-1}$ in $\Omega$. For $t \in [0, 1]$ this holds 
by the definition of $f$. If $t>1$, then by \aoneo{\Omega} for $x, y\in \Omega$ we have
\[
\beta^{-|x-y| t^{1/n} } \phi^{-1}(y, t)
\ge \beta^{-|x-y| t^{1/n} } \beta^{|x-y| t^{1/n} +1}  \phi^{-1}(x, t)
= \beta \phi^{-1}(x, t).
\]
Since $(\phi^-_\Omega)^{-1} (t) \ge \phi^{-1}(x, t)$, this yields that 
$f(x,t) \ge \beta \phi^{-1}(x, t)$. On the other hand by \aoneo{\Omega} we have for $y\in \hat\Omega$ that 
\[
f(x, t) \le \beta^{-|x-y| t^{1/n} } \phi^{-1}(y, t) \le \beta^{-|x-y| t^{1/n} } \beta^{-|x-y| t^{1/n} -1}  \phi^{-1}(x, t) \to \tfrac1{\beta} \phi^{-1}(x, t)
\]
as $y \to x$ and hence $f(x, t) \le  \frac1{\beta} \phi^{-1}(x, t)$.


\underline{Claim \ref{f8}:}
Let us then prove \aone{}. 
The required inequality for $t=1$ follows from \azero{}, so we take  $t \in (1, \frac{1}{|B|}]$ 
and $x, y \in B$. Note that $|x-y|t^{1/n} \le c(n)$. 
Let $z_0 \in \hat \Omega$ be such that $\inf_{z\in \hat\Omega} \beta^{-|x-z| t^{1/n} } \phi^{-1}(z, t) \ge \frac 12  \beta^{-|x-z_0| t^{1/n} } \phi^{-1}(z_0, t)$. Then 
by the definition of infimum, the triangle inequality, and the choice of $z_0$ we obtain that 
\[
\begin{split}
\inf_{z\in \hat\Omega} \beta^{-|x-z| t^{1/n} } \phi^{-1}(z, t) 
&\le  \beta^{-|y-z_0| t^{1/n} } \phi^{-1}(z_0, t)
\le  \beta^{-(|y-x| + |x-z_0|) t^{1/n} } \phi^{-1}(z_0, t)\\
&\le 2 \beta^{- c(n)} \inf_{z\in \hat\Omega} \beta^{-|x-z| t^{1/n} } \phi^{-1}(z, t).
\end{split}
\]
The part $(\phi^-_\Omega)^{-1}$ satisfies \aone{} since it is independent of $x$. 
A short calculation shows that the minimum of two functions that satisfy \aone{} 
also satisfies it.


\underline{Claim \ref{f10}:} 
Assume that $\phi$ satisfies \adec{q}. Then $\phi^{-1}$ and 
$\phi_\infty^{-1}$ satisfy \ainc{1/q}, which imply the condition for $f$ when $t \in [0, 1]$. 
For $1<t<s$, we use the condition for $\phi^{-1}$ as well as 
$\beta^{-|x-y| t^{1/n} } \le \beta^{-|x-y| s^{1/n} }$ to conclude that 
\[
\begin{split}
\frac{f(x, t)}{t^{1/q}} &= \min \Big\{ \frac{(\phi^-_\Omega)^{-1}(t)}{t^{1/q}}, \inf_{y\in \hat\Omega} \beta^{-|x-y| t^{1/n} } \frac{\phi^{-1}(y, t)}{t^{1/q}} \Big\}\\
&\lesssim  \min \Big\{ \frac{(\phi^-_\Omega)^{-1}(s)}{s^{1/q}}, \inf_{y\in \hat\Omega} \beta^{-|x-y| t^{1/n} } \frac{\phi^{-1}(y, s)}{s^{1/q}} \Big\}
\le \frac{f(x, s)}{s^{1/q}}.
\end{split}
\]
The case $1=t<s$ follows as $\epsilon\to 0^+$ since $f$ is increasing:
\[
\frac{f(x, 1)}{1^{1/q}} 
\le 
(1+\epsilon)^{1/q} \frac{f(x, 1+\epsilon)}{(1+\epsilon)^{1/q}} 
\lesssim 
(1+\epsilon)^{1/q} \frac{f(x, s)}{s^{1/q}}.
\]

\bigskip
The function $f$ does not satisfy \adec{1}, so it is not the inverse of any $\Phi$-function. 
We therefore make a regularization to ensure this growth condition. 
Recall that $\phi$ satisfies \ainc{p} for some $p\ge 1$. 
We define $g(x, 0):=0$, $g(x,\infty):= \infty$ and
\[
g(x, t) := t^{1/p} \inf_{0<s \le t} \frac{f(x, s)}{s^{1/p}}.
\]
Since $f$ satisfies \adec{1/p} on $[0,1]$, we have 
\begin{equation}\label{equ:g-pienet-arvot}
t^{1/p} \inf_{0<s \le t} \frac{f(x, s)}{s^{1/p}} \approx f(x,t)
\end{equation}
for $t\in[0,1]$, so that $g\approx f$ in the same range of $t$. 
From the corresponding properties of $f$ we conclude that 
$g$ is left-continuous and satisfies \azero{}. 

Let us prove the following properties for $g$:
\begin{enumerate}[label=(\roman*)]
\item\label{g4} $g$ is measurable. 
\item\label{g5} $g$ satisfies \dec{1/p}.
\item\label{g6} $g|_\Omega \approx \phi^{-1}$. 
\item\label{g8} $g$ satisfies \aone{}. 
\item\label{g1} $g(x, t)=0$ iff $t=0$, and $g(x, t)=\infty$ iff $t=\infty$. 
\item\label{g10} $g$ is increasing; and $g$ satisfies \ainc{1/q} provided that $\phi$ satisfies \adec{q}.
\end{enumerate}

%


\underline{Claim \ref{g4}:} 
Since $f$ is left-continuous, we obtain that
\[
g(x, t) = t^{1/p} \inf_{s\in(0, t] \cap Q} \frac{f(x,s)}{s^{1/p}}.
\]
Since $x \mapsto f(x, t)$ is measurable, this implies that $x \mapsto g(x, t)$ 
is measurable as the infimum of countable many measurable functions.

\underline{Claim \ref{g5}:} 
Let us next show that $g$ satisfies \dec{1/p}. For that let $0<t<\tau$. By the definition of $g$ we obtain
\[
\frac{g(x,t)}{t^{1/p}} =  \inf_{0<s \le t} \frac{f(x, s)}{s^{1/p}} \ge \inf_{0<s\le \tau} \frac{f(x, s)}{s^{1/p}}
= \frac{g(x,\tau)}{\tau^{1/p}}.
\]

\underline{Claim \ref{g6}:}
Since $f \approx \phi^{-1}$ in $\Omega$ and $\phi^{-1}$ satisfies \adec{1/p}, 
we obtain $g|_\Omega\approx \phi^{-1}$ in the same way as in \eqref{equ:g-pienet-arvot}.


\underline{Claim \ref{g8}:} 
Then we show that $g$ satisfies \aone{}. Let $B$ be a ball with $|B| \le 1$, 
$t \in [1, \frac1{|B|}]$ and $x, y \in \Omega$. Since $f$ satisfies \adec{1/p} in $(0, 1]$, 
we have
\begin{equation}\label{equ:g_alara}
g(y,t) \approx  t^{1/p}\inf_{1\le s \le t} \frac{f(y, s)}{s^{1/p}}. 
\end{equation}
Thus by \aone{} of $f$ we obtain
\[
g(x, t) \approx t^{1/p}\inf_{1\le s \le t} \frac{f(x, s)}{s^{1/p}} \lesssim  t^{1/p}\inf_{1\le s \le t} \frac{f(y, s)}{s^{1/p}}
\approx g(y, t).
\]

\underline{Claim \ref{g1}:} 
Using the inequality $f(x,1)\le f(x,s)\le f(x,t)$ and \azero{}, we obtain by \eqref{equ:g_alara} 
for $1<t$, that
\[
0< \beta_0 \le  g(x, t) \le t^{1/p} f(x, t) < \infty.
\]
By \eqref {equ:g-pienet-arvot}  we have $f\approx g$ for $s\in [0,1]$.
Thus by the corresponding property of $f$ we have 
$g(x, t)=0$ if and only if $t=0$, and $g(x, t)=\infty$ if and only if  $t=\infty$.


\underline{Claim \ref{g10}:}
Assume that $f$ satisfies \ainc{1/q} with constant $L\ge 1$. Let $0<t<s$, $\epsilon >0$ 
and choose $\theta \in (0,1]$ such that 
$g(x, s) \ge \theta^{-1/p} f(x,\theta s)- \epsilon$. Then 
by the definition of $g$, \adec{1/q} of $f$ and the choice of $\theta$, we obtain that 
\[
\frac{g(x, t)}{t^{1/q}} 
\le t^{1/p -1/q} \frac{ f(x, \theta t)}{(\theta t)^{1/p}} 
= \theta^{1/q-1/p} \frac{ f(x, \theta t)}{(\theta t)^{1/q}}
\le L\theta^{1/q-1/p} \frac{ f(x, \theta s)}{(\theta s)^{1/q}} 
\le L\frac{g(x, s)+ \epsilon}{s^{1/q}}.
\]
Thus letting $\epsilon \to 0^+$, we find that $g$ satisfies \ainc{1/q}.
If we take $L=1$ and $1/q=0$, this implies that $g$ is increasing (since $f$ is increasing).

\bigskip

We set $\psi:= g^{-1}$. 
Using \ref{g1}--\ref{g5}, increasing from \ref{g10} and Proposition~2.5.14 in \cite{HarH_book}  we obtain that 
$\psi \in \Phiw(\Rn)$ and  $\psi^{-1} = (g^{-1})^{-1}= g$. 
Since $\phi^{-1} \approx g= \psi^{-1}$, it follows that $\psi|_\Omega \simeq \phi$ 
\cite[Theorem~2.3.6]{HarH_book}. 
Now properties \azero{} and \aone{} for $\psi$ follow. 
Moreover, if $\phi$ satisfies \adec{q}, then by \ref{g10} $\psi$ satisfies it as well.

When $x\in \Rn\setminus \Omega$ and $t\in [0,1]$, we have 
$g(x,t)\approx f(x,t)=\beta_0^2\phi_\infty^{-1}(t)$. It follows that 
$\psi(x,t)\simeq \phi_\infty(t)$ for sufficiently small values of $t$ and $x$ as before. 
Therefore, $\psi$ satisfies the condition of \eqref{equ:A2'}, and 
so $\psi$ satisfies \atwo{}.
\end{proof}


\bibliographystyle{amsplain}

\end{document}